\documentclass[reqno]{amsart}

\usepackage{xypic,hyperref,todonotes,comment}
\usepackage{graphics,amssymb}
\usepackage{pagecolor,lipsum}
\usepackage{latexsym}
\usepackage{mathrsfs}
\xyoption{curve}
\usepackage{hyperref}

\usepackage{color}
\definecolor{color1}{HTML}{269900}
\definecolor{color2}{rgb}{.20,.60,.22}
\definecolor{color3}{rgb}{0,.40,.80}

\hypersetup{
colorlinks,
citecolor=color1,
filecolor=blue,
linkcolor=black,
urlcolor=color1
}

\newcommand{\nocontentsline}[3]{}
\newcommand{\tocless}[2]{\bgroup\let\addcontentsline=\nocontentsline#1{#2}\egroup}

\newtheorem{lemma}{Lemma}[section]
\newtheorem{proposition}[lemma]{Proposition}
\newtheorem{theorem}[lemma]{Theorem}
\newtheorem{corollary}[lemma]{Corollary}

\newtheorem*{theoremA}{Theorem}

\theoremstyle{definition}

\newtheorem{definition}[lemma]{Definition}
\newtheorem{remark}[lemma]{Remark}
\newtheorem{notation}[lemma]{Notation}

\newcommand{\mfk}[1]{\mathfrak{#1}}
\newcommand{\mbb}[1]{\mathbb{#1}}
\newcommand{\mcl}[1]{\mathcal{#1}}

\newcommand{\msc}[1]{\mathscr{#1}}

\DeclareMathOperator{\Hom}{Hom}
\DeclareMathOperator{\End}{End}
\DeclareMathOperator{\RHom}{RHom}
\DeclareMathOperator{\REnd}{REnd}
\DeclareMathOperator{\Ext}{Ext}

\DeclareMathOperator{\rep}{rep}
\DeclareMathOperator{\Rep}{Rep}

\DeclareMathOperator{\Spec}{Spec}
\DeclareMathOperator{\Kdim}{Kdim}
\DeclareMathOperator{\Spf}{Spf}
\DeclareMathOperator{\res}{res}
\DeclareMathOperator{\Coh}{Coh}
\DeclareMathOperator{\Sym}{Sym}

\newcommand{\RG}{\mathrm{R}G}
\newcommand{\ot}{\otimes}
\newcommand{\K}{\mathcal{K}}

\renewcommand{\1}{\mathbf{1}}
\renewcommand{\O}{\mathscr{O}}
\renewcommand{\hat}{\widehat}

\renewcommand{\b}[1]{[\!\hspace{.1mm}[{#1}]\!\hspace{.1mm}]}

\title[]{Finite generation of cohomology for Drinfeld doubles of finite group schemes}

\author{Cris Negron}
\address{Department of Mathematics, University of North Carolina, Chapel Hill, NC 27599}
\email{cnegron@email.unc.edu}

\begin{document}
\maketitle

\begin{abstract}
We prove that the Drinfeld double of an arbitrary finite group scheme has finitely generated cohomology.  That is to say, for $G$ any finite group scheme, and $D(G)$ the Drinfeld double of the group ring $kG$, we show that the self-extension algebra of the trivial representation for $D(G)$ is a finitely generated algebra, and that for each $D(G)$-representation $V$ the extensions from the trivial representation to $V$ form a finitely generated module over the aforementioned algebra.  As a corollary, we find that all categories $\rep(G)^\ast_\msc{M}$ dual to $\rep(G)$ are of also of finite type (i.e.\ have finitely generated cohomology), and we provide a uniform bound on their Krull dimensions.  This paper completes earlier work of E.\ M.\ Friedlander and the author.
\end{abstract}

\tocless\section{Introduction}
\label{sect:intro}

Fix $k$ an arbitrary field of finite characteristic.  Let us recall some terminology~\cite{negronplavnik}: A finite $k$-linear tensor category $\msc{C}$ is said to be of {\it finite type (over $k$)} if the self-extensions of the unit object $\Ext^\ast_\msc{C}(\1,\1)$ are a finitely generated $k$-algebra, and for any object $V$ in $\msc{C}$ the extensions $\Ext^\ast_\msc{C}(\1,V)$ are a finitely generated module over this algebra.  In this case, the {\it Krull dimension} $\Kdim\msc{C}$ of $\msc{C}$ is the Krull dimension of the extension algebra of the unit.  One is free to think of $\msc{C}$ here as the representation category $\rep(A)$ of a finite-dimensional Hopf algebra $A$, with monoidal structure induced by the comultiplication, and unit $\1=k$ provided by the trivial representation.
\par

It has been conjectured \cite[Conjecture 2.18]{etingofostrik04} \cite{friedlandersuslin97} that any finite tensor category, over an arbitrary base field, is of finite type.  Here we consider the category of representations for the Drinfeld double $D(G)$ of a finite group scheme $G$, which is identified with the so-called Drinfeld center $Z(\rep(G))$ of the category of finite $G$-representations~\cite{muger03,egno15}.  The Drinfeld double $D(G)$ is the smash product $\O(G)\rtimes kG$ of the algebra of global functions on $G$ with the group ring $kG$, under the adjoint action.  So, one can think of $Z(\rep(G))$, alternatively, as the category of coherent $G$-equivariant sheaves on $G$ under the adjoint action
\[
Z(\rep(G))=\rep(D(G))=\Coh(G)^G.
\]
\par

In the present work we prove the following.

\begin{theoremA}[\ref{thm:ZG}]
For any finite group scheme $G$, the Drinfeld center $Z(\rep(G))$ is of finite type and of Krull dimension
\[
\Kdim Z(\rep(G))\leq \Kdim\rep(G)+\operatorname{embed.dim}(G).
\]
\end{theoremA}

Here $\operatorname{embed.dim}(G)$ denotes the minimal dimension of a smooth (affine) algebraic group in which $G$ embeds as a closed subgroup.  The above theorem was proved for $G=\mbb{G}_{(r)}$ a Frobenius kernel in a smooth algebraic groups $\mbb{G}$ in work of E.\ M.\ Friedlander and the author~\cite{friedlandernegron18}.  Thus Theorem~\ref{thm:ZG} completes, in a sense, the project of~\cite{friedlandernegron18}.
\par

One can apply Theorem~\ref{thm:ZG}, and results of J.\ Plavnik and the author~\cite{negronplavnik}, to obtain an additional finite generation result for all dual tensor categories $\rep(G)^\ast_\msc{M}(:=\End_{\rep(G)}(\msc{M}))$, calculated relative to an exact $\rep(G)$-module category $\msc{M}$ \cite[Section 3.3]{etingofostrik04}.

\begin{corollary}\label{cor:ZGM}
Let $G$ be a finite group scheme, and $\msc{M}$ be an arbitrary exact $\rep(G)$-module category.  Then the dual category $\rep(G)^\ast_\msc{M}$ is of finite type and of uniformly bounded Krull dimension
\[
\Kdim\rep(G)^\ast_\msc{M}\leq \Kdim\rep(G)+\operatorname{embed.dim}(G).
\]
\end{corollary}

\begin{proof}
Immediate from Theorem~\ref{thm:ZG} and~\cite[Corollary 4.1]{negronplavnik}.
\end{proof}

We view Theorem~\ref{thm:ZG}, and Corollary~\ref{cor:ZGM}, as occurring in a continuum of now very rich studies of cohomology for finite group schemes, e.g.~\cite{friedlanderparshall86II,friedlandersuslin97,suslinfriedlanderbendel97,friedlanderpevtsova07,touzevanderkallen10,drupieski16,bikp19}.

\begin{remark}
Exact $\rep(G)$-module categories have been classified by Gelaki~\cite{gelaki15}, and correspond to pairs $(H,\psi)$ of a subgroup $H\subset G$ and certain $3$-cocycle $\psi$ which introduces an associativity constraint for the action of $\rep(G)$ on $\rep(H)$.
\end{remark}

\begin{remark}
For an analysis of support theory for Drinfeld doubles of some solvable height $1$ group schemes, one can see \cite{negronpevtsova,negronpevtsovaII}.  The problem of understanding support for general doubles $D(G)$ is, at this point, completely open.
\end{remark}

\subsection{Approach via equivariant deformation theory}

In~\cite{friedlandernegron18}, where the Frobenius kernel $\mbb{G}_{(r)}$ in a smooth algebraic group $\mbb{G}$ is considered, we basically use the fact that ambient group $\mbb{G}$ provides a smooth, equivariant, deformation of $\mbb{G}_{(r)}$ parametrized by the quotient $\mbb{G}/\mbb{G}_{(r)}\cong \mbb{G}^{(r)}$ in order to gain a foothold in our analysis of cohomology.  In particular, the adjoint action of $\mbb{G}_{(r)}$ on $\mbb{G}$ descends to a trivial action on the twist $\mbb{G}^{(r)}$, so that the Frobenius map $\mbb{G}\to \mbb{G}^{(r)}$ can be viewed as smoothly varying family of $\mbb{G}_{(r)}$-algebras which deforms the algebra of functions $\O(\mbb{G}_{(r)})$.  Such a deformation situation provides ``deformation classes" in degree $2$,
\[
\{\text{deformation classes}\}=T_1\mbb{G}^{(r)}\subset \Ext^2_{\Coh(\mbb{G}_{(r)})^{\mbb{G}_{(r)}}}(\1,\1)=\Ext^2_{D(\mbb{G}_{(r)})}(\1,\1).
\]
One uses these deformation classes, in conjunction with work of Friedlander and Suslin \cite{friedlandersuslin97}, to find a finite set of generators for extensions.
\par

For a general finite group scheme $G$, we can try to pursue a similar deformation approach, where we embed $G$ into a smooth algebraic group $\mcl{H}$, and consider $\mcl{H}$ as a deformation of $G$ parametrized by the quotient $\mcl{H}/G$.  However, a general finite group scheme may not admit any \emph{normal} embedding into a smooth algebraic group.  (This is the case for certain non-connected finite group schemes, and should also be the case for restricted enveloping algebras $kG=u^{\rm res}(\mfk{g})$ of Cartan type simple Lie algebras, for example).  So, in general, one accepts that $G$ acts nontrivially on the parametrization space $\mcl{H}/G$, and that the fibers in the family $\mcl{H}$ are permuted by the action of $G$ here.  Thus we do not obtain a smoothly varying family of $G$-algebras deforming $\O(G)$ in this manner.
\par

One can, however, consider a \emph{type} of equivariant deformation theory where the group $G$ is allowed to act nontrivially on the parametrization space, and attempt to obtain \emph{higher} deformation classes in this instance
\[
\{\text{higher deformation classes}\}\subset \Ext^{\geq 2}_{\Coh(G)^G}(\1,\1)=\Ext_{D(G)}^{\geq 2}(\1,\1).
\]
We show in Sections~\ref{sect:equiv1} and~\ref{sect:equiv2} that this equivariant deformation picture can indeed be formalized, and that--when considered in conjunction with work of Touz\'{e} and Van der Kallen \cite{touzevanderkallen10}--it can be used to obtain the desired finite generation results for cohomology (see in particular Theorems~\ref{thm:equiv_fg} and \ref{thm:tv2}).

\begin{remark}
From a geometric perspective, one can interpret our main theorem as a finite generation result for the cohomology of non-tame stacky local complete intersections.  (Formally speaking, we only deal with the maximal codimension case here, but the general situation is similar.)  One can compare with works of Gulliksen~\cite{gulliksen74}, Eisenbud~\cite{eisenbud80}, and many others regarding the homological algebra of complete intersections.
\end{remark}

\subsection{Acknowledgements}

Thanks to Ben Briggs, Christopher Drupieski, Eric Friedlander, Julia Pevtsova, Antoine Touz\'{e}, and Sarah Witherspoon for helpful conversations.  The proofs of Lemmas \ref{lem:231} and \ref{lem:well_def} are due to Ben Briggs and Ragnar Buchweitz (with any errors in their reproduction due to myself).  This material is based upon work supported by the National Science Foundation under Grant No.\ DMS-1440140, while the author was in residence at the Mathematical Sciences Research Institute in Berkeley, California, during the Spring 2020 semester.

\section{Differential generalities}

Throughout $k$ is a field of finite characteristic, which is not necessarily algebraically closed.  Schemes and algebras are $k$-schemes and $k$-algebras, and $\ot=\ot_k$.  All group schemes are \emph{affine} group schemes which are of finite type over $k$, and throughout $G$ denotes an (affine) group scheme.

\subsection{Commutative algebras and modules}

A \emph{finite type} commutative algebra over a field $k$ is a finitely generated $k$-algebra.  A \emph{coherent} module over a commutative Noetherian algebra is a finitely generated module.  We adopt this language, at times, to distinguish clearly between these two notions of finite generation.

\subsection{$G$-equivariant dg algebras}

Consider $G$ an affine group scheme.  We let $\rep(G)$ denote the category of finite-dimensional $G$-representations, $\Rep(G)$ denote the category of integrable, i.e.\ locally finite, representations, and $\operatorname{Ch}(\Rep(G))$ denote the category of cochain complexes over $\Rep(G)$.  Each of these categories is considered along with its standard monoidal structure.
\par

By a $G$-algebra we mean an algebra object in $\Rep(G)$, and by a dg $G$-algebra we mean an algebra object in $\operatorname{Ch}(\Rep(G))$.  For $T$ any commutative $G$-algebra, by a \emph{$G$-equivariant dg $T$-algebra} $S$ we mean a $T$-algebra in $\operatorname{Ch}(\Rep(G))$.  Note that, for such a dg algebra $S$, the associated sheaf $S^\sim$ on $\Spec(T)$ is an equivariant sheaf of dg algebras, and vice versa.  Note also that a dg $G$-algebra is the same thing as an equivariant dg algebra over $T=k$.

\subsection{DG modules and resolutions}

For $S$ a dg $G$-algebra, we let $S\text{-dgmod}^G$ and $D(S)^G$ denote the category of $G$-equivariant dg modules over $S$ and its corresponding derived category $D(S)^G=(S\text{-dgmod}^G)[{\rm quis}^{-1}]$.  (Of course, by an equivariant dg module we mean an $S$-module in the category of cochains over $G$.)  If we specify some commutative Noetherian graded $G$-algebra $T$, and equivariant $T$-algebra structure on cohomology $T\to H^\ast(S)$, then we take
\[
D_{coh}(S)^G:=\left\{\begin{array}{c}
\text{The full subcategory in $D(S)^G$ consisting of dg modules}\\
\text{$M$ with finitely generated cohomology over $T$}
\end{array}\right\}.
\]
When $T=k$ we take $D_{fin}(S)^G=D_{coh}(S)^G$.
\par

A (non-equivariant) \emph{free} dg $S$-module is an $S$-module of the form $\oplus_{j\in J}\Sigma^{n_j}S$, where $J$ is some indexing set.  A semi-projective resolution of a (non-equivariant) dg $S$-module $M$ is a quasi-isomorphism $F\to M$ from a dg module $F$ equipped with a filtration $F=\cup_{i\geq 0} F(i)$ by dg submodules such that each subquotient $F(i)/F(i-1)$ is a summand of a free $S$-module.   An \emph{equivariant semi-projective resolution} of an equivariant dg module $M$ is a $G$-linear quasi-isomorphism $F\to M$ from an equivariant dg module $F$ which is non-equivariantly semi-projective.  The usual shenanigans, e.g.~\cite[Lemma 13.3]{drinfeld04}, shows that equivariant semi-projective resolutions always exist.

\subsection{Homotopy isomorphisms}
\label{sect:dg_htop}

Consider $S$ and $A$ dg $G$-algebras, over some given group scheme $G$.  By an \emph{(equivariant) homotopy isomorphism} $f:S\to A$ we mean a zig-zag of $G$-linear dg algebra quasi-isomorphism
\begin{equation}\label{eq:207}
S\overset{\sim}\leftarrow S_1\overset{\sim}\to S_2\cdots\overset{\sim}\leftarrow S_{N-1}\overset{\sim}\to A.
\end{equation}
\par

We note that we use the term \emph{homotopy} informally here, as we do not propose any particular model structure on the category of dg $G$-algebras (cf.\ \cite{tabuada05,toen07}).  Throughout the text, when we speak of homotopy isomorphisms between dg $G$-algebras we always mean equivariant homotopy isomorphisms.
\par

A homotopy isomorphism $f:S\to A$ as in \eqref{eq:207} specifies a triangulated equivalence between the corresponding derived categories of dg modules
\begin{equation}\label{eq:211}
f_\ast:D(S)^G\overset{\sim}\to D(A)^G,
\end{equation}
via successive application of base change and restriction along the maps to/from the $S_i$.  To elaborate, an equivariant quasi-isomorphism $f:S_1\to S_2$ specifies mutually inverse equivalences $S_2\ot_{S_1}^{\rm L}-:D(S_1)^G\to D(S_2)^G$ and $\operatorname{res}_f:D(S_2)^G\to D(S_1)$.  So for a homotopy isomorphism $f:S\to A$, compositions of restriction and base change produce the equivalence \eqref{eq:211}.
\par

Note that, on cohomology, such a homotopy isomorphism $f:S\to A$ induces an actual isomorphism of algebras $H^\ast(f):H^\ast(S)\to H^\ast(A)$, and one can check that for a dg module $M$ over $S$ we have
\[
H^\ast(f_\ast M)\cong H^\ast(A)\ot_{H^\ast(S)}H^\ast(M)\cong\res_{H^\ast(f)^{-1}}H^\ast(M).
\]
So, in particular, if $H^\ast(S)$ and $H^\ast(A)$ are $T$-algebras, for some commutative Noetherian $T$, and $H^\ast(f)$ is $T$-linear, then the equivalence \eqref{eq:211} restricts to an equivalence
\[
f_\ast:D_{coh}(S)^G\overset{\sim}\to D_{coh}(A)^G
\]
between the corresponding equivariant, coherent, derived categories.

\begin{definition}
We say a dg $G$-algebra $S$ is \emph{equivariantly formal} if $S$ is equivariantly homotopy isomorphic to its cohomology $H^\ast(S)$.
\end{definition}

\subsection{Derived maps and derived endomorphisms}

Fix $S$ a dg $G$-algebra, over a group scheme $G$.  For such $S$, the dg $\Hom$ functor $\Hom_S$ on $S\text{-dgmod}^G$ naturally takes values in $\operatorname{Ch}(\Rep(G))$.  Namely, for $x$ in the group ring $kG=\O(G)^\ast$, we act on functions $f\in \Hom_S(M,N)$ according to the formula
\[
(x\cdot f)(m):=x_1 f(S(x_2)m).
\]
With these actions each $\Hom_S(M,N)$ is a dg $G$-representation, and composition
\[
\Hom_S(N,L)\ot\Hom_S(M,N)\to \Hom_S(M,L)
\]
is a map of dg $G$-representations.  In particular, $\End_S(M)$ is a dg $G$-algebra for any equivariant dg module $M$ over $S$.

\begin{remark}
One needs to use cocommutativity of $kG$ here to see that $x\cdot f$ is in fact $S$-linear for $S$-linear $f$.
\end{remark}

We derive the functor $\Hom_S$ to $\operatorname{Ch}(\Rep(G))$ by taking 
\[
\RHom_S(M,N):=\Hom_S(M',N),
\]
where $M'\to M$ is any equivariant semi-projective resolution of $M$.  One can apply their favorite arguments to see that $\RHom_S(M,N)$ is well-defined as an object in $D(\Rep(G))$, or refer to the following lemma.

\begin{lemma}\label{lem:231}
For any two equivariant resolutions $M_1\to M$ and $M_2\to M$ there is an equivariant semi-projective dg module $F$ which admits two surjective, equivariant, quasi-isomorphisms $F\to M_1$ and $F\to M_2$.
\end{lemma}

\begin{proof}
By adding on acyclic semi-projective summands we may assume that the given maps $f_i:M_i\to M$ are surjective.  For example, one can take a surjective resolution $N\to M$, consider the mapping cone $\operatorname{cone}(id_N)$, then replaces the $M_i$ with $(\Sigma^{-1}\operatorname{cone}(id_N))\oplus M_i$ .  So, let us assume that the $f_i$ here are surjective.
\par

We consider now the fiber product $F_0$ of the maps $f_1$ and $f_2$ to $M$.  Note that the structure maps $F_0\to M_i$ are surjective, since the $f_i$ are surjective.  We have the exact sequence
\[
0\to F_0\to M_1\oplus M_2\overset{[f_1\ -f_2]^T}\to M\to 0
\]
and by considering the long exact sequence on cohomology find that we have also an exact sequence
\[
0\to H^\ast(F_0)\to H^\ast(M_1)\oplus H^\ast(M_2)\to H^\ast(M)\to 0,
\]
with the map from $H^\ast(M_1)\oplus H^\ast(M_2)$ the sum of isomorphisms $\pm H^\ast(f_i)$.  It follows that the composites $H^\ast(F_0)\to H^\ast(M_1)\oplus H^\ast(M_2)\to H^\ast(M_i)$ are both isomorphisms, and hence that the maps $F_0\to M_1$ and $F_0\to M_2$ are quasi-isomorphisms.  One considers $F\to F_0$ any surjective, equivariant, semi-projective resolution to obtain the claimed result.
\end{proof}

For $M$ in $D(S)^G$ we take $\REnd_S(M)=\End_S(M')$, for $M'\to M$ any equivariant semi-projective resolution.  The following result should be known to experts.  The proof we offer is due to Benjamin Briggs and Ragnar Buchweitz.  I thank Briggs for communicating the proof to me, and allowing me to repeat it here.

\begin{lemma}\label{lem:well_def}
$\REnd_S(M)$ is well-defined, as a dg $G$-algebra, up to homotopy isomorphism.  Furthermore, if $M$ and $N$ are isomorphic in $D(S)^G$, then $\REnd_S(M)$ and $\REnd_S(N)$ are homotopy isomorphic as well.
\end{lemma}

Given an explicit isomorphism $\xi:M\to N$ in $D(S)^G$, the homotopy isomorphism $\RHom_S(M)\to \RHom_S(N)$ can in particular be chosen to lift the canonical isomorphism $\operatorname{Ad}_\xi:\Ext^\ast_S(M,M)\to \Ext^\ast_S(N,N)$ on cohomology.

\begin{proof}
Consider two equivariant semi-projective resolutions $M_1\to M$ and $M_2\to M$.  By Lemma \ref{lem:231} we may assume that the map $M_1\to M$ lifts to a surjective, equivariant, quasi-isomorphism $f:M_1\to M_2$.  In this case we have the two quasi-isomorphisms $f_\ast$ and $f^\ast$ of $\Hom$ complexes, and consider the fiber product
\begin{equation}\label{eq:261}
\xymatrixrowsep{3mm}
\xymatrix{
	& B\ar@{-->}[dl]\ar@{-->}[dr] & \\
\End_S(M_1)\ar[dr]_{f_\ast} & & \End_S(M_2)\ar[dl]^{f^\ast}\\
	&\Hom_S(M_1,M_2)
}
\end{equation}
As $f_\ast$ and $f^\ast$ are maps of dg $G$-representations, $B$ is a dg $G$-representation.  Furthermore, one checks directly that $B$ is a dg algebra, or more precisely a dg subalgebra in the product $\End(M_1)\times \End(M_2)$.  So the top portion of \eqref{eq:261} is a diagram of maps of dg $G$-algebras.
\par

As $M_1$ is projective, as a non-dg module, the map $f_\ast$ is a \emph{surjective} quasi-isomorphism.  One can therefore argue as in the proof of Lemma \ref{lem:231} to see that the structure maps from $B$ to the $\End_S(M_i)$ are quasi-isomorphisms.  So we have the explicit homotopy isomorphism
\[
\End_S(M_1)\overset{\sim}\leftarrow B\overset{\sim}\to \End_S(M_2).
\]

Now, if $M$ is isomorphic to $N$ in $D(S)^G$, then there is a third equivariant dg module $\Omega$ with quasi-isomorphisms $M\overset{\sim}\leftarrow \Omega\overset{\sim}\to N$.  Any resolution $F\overset{\sim}\to \Omega$ therefore provides a simultaneous resolution of $M$ and $N$, and we may take $\REnd_S(M)=\End_S(F)=\REnd_S(N)$.
\end{proof}

\section{Equivariant deformations and Koszul resolutions}
\label{sect:equiv1}

In Sections~\ref{sect:equiv1} and~\ref{sect:equiv2} we develop the basic homological algebra associated with equivariant deformations.  Our main aim here is to provide equivariant versions of results of Bezrukavnikov and Ginzburg~\cite{bezrukavnikovginzburg07}, and Pevtsova and the author~\cite[\S 4]{negronpevtsova} (cf.~\cite{eisenbud80,arkhipovbezrukavnikovginzburg04}).

\subsection{Equivariant deformations}

We recall that a deformation of an algebra $R$, parametrized by an augmented commutative algebra $Z$, is a choice of flat $Z$-algebra $Q$ along with an algebra map $Q\to R$ which reduces to an isomorphism $k\ot_Z Q\cong R$.  We call such a deformation $Q\to R$ an \emph{equivariant deformation} if all of the algebras present are $G$-algebras, and all of the structure maps $Z\to Q$, $Z\to k$, and $Q\to R$ are maps of $G$-algebras.
\par

The interesting point here, and the point of deviation with other interpretations of equivariant deformation theory, is that we allow $G$ to act nontrivially on the parametrization space $\Spec(Z)$ (or $\Spf(Z)$ in the formal setting).

\subsection{An equivariant Koszul resolution}
\label{sect:KZ}

We fix a group scheme $G$, and equivariant deformation $Q\to R$ of a $G$-algebra $R$ with formally smooth parametrization space space $\Spf(Z)$.  We require specifically that $Z$ is isomorphic to a power series $k\b{x_1,\dots,x_n}$ in finitely many variables.  As the distinguished point $1\in \Spf(Z)$ is a fixed point for the $G$-action, the cotangent space $T_1\Spf(Z)=m_Z/m_Z^2$ admits a natural $G$-action, and so does the graded algebra 
\[
\Sym(\Sigma m_Z/m_Z^2)=\wedge^\ast(m_Z/m_Z^2),
\]
which we view as a dg $G$-algebra with vanishing differential.

\begin{lemma}[{cf.~\cite[Lemma 5.1.4]{arkhipovbezrukavnikovginzburg04}}]\label{lem:equiv_kos}
One can associate to the parametrization algebra $Z$ a commutative equivariant dg $Z$-algebra $\K_Z$ such that
\begin{enumerate}
\item $\mcl{K}_Z$ is finite and flat over $Z$, and
\item $\mcl{K}_Z$ admits quasi-isomorphisms $\K_Z\overset{\sim}\to k$ and $k\ot_Z \K_Z\overset{\sim}\to \Sym(\Sigma m_Z/m_Z^2)$ of equivariant dg algebras.
\end{enumerate}
\end{lemma}

\begin{proof}[Construction]
We first construct an unbounded dg resolution $\K'$ of $k$, as in \cite[Section 2.6]{ciocankapranov01}, then truncate to obtain $\K$.  We construct $\K'$ as a union $\K'=\varinjlim_{i\geq 0} \K(i)$ of dg subalgebras $\K(i)$ over $Z$.  We define the $\K(i)$ inductively as follows: Take $\K(0)=Z$ and, for $V_1$ a finite-dimensional $G$-subspace generating the maximal ideal $m_Z$ in $Z$, we take $\K(1)=\Sym_Z(Z\ot \Sigma V_1)$ with differential $d(\Sigma v)=v$, $v\in V_1$.  
\par

Suppose now that we have $\K(i)$ an equivariant dg algebra which is finite and flat over $Z$ in each degree, and has (unique) augmentation $\K(i)\to k$ which is a quasi-isomorphism in degrees $> -i$.  Let $V_i$ be an equivariant subspace of cocycles in $\K(i)^{-i}$ which generates $H^{-i}(\K(i))$, as a $Z$-module.  Define
\[
\K(i+1)=\Sym_Z(Z\ot \Sigma V_i)\ot_Z \K(i),\ \text{with extended differential $d(\Sigma v)=v$ for }v\in V_i.
\]
We then have the directed system of dg algebras $\K(0)\to \K(1)\to \dots$ with colimit $\K'=\varinjlim_i \K(i)$.  By construction $\K'$ is finite and flat over $Z$ in each degree, and has cohomology $H^\ast(\K')=k$.
\par

Since $Z$ is of finite flat dimension, say $n$, the quotient
\[
(\K_Z:=)\K=\K'/((\K')^{<-n}+B^{-n}(\K'))
\]
is finite and flat over $Z$ in all degrees.  Furthermore, $\K$ inherits a $G$-action so that the quotient map $\K'\to \K$ is an equivariant quasi-isomorphism.  So we have produced a finite flat dg $Z$-algebra $\K$ with equivariant quasi-isomorphism $\K\overset{\sim}\to k$.
\par

We consider a section $m_Z/m_Z^2\to V_1$ of the projection $V_1\to m_Z/m_Z^2$, and let $\bar{S}_1\subset V_1$ denote the image of this section.  Take $S=\Sym_Z(Z\ot \Sigma \bar{S}_1)$ with differential specified by $d(\Sigma v)=v$ for $v\in \bar{S}_1$.  Then $S$ the the standard Koszul resolution for $k$, and the inclusion $S\to \K$ is a (non-equivariant) dg algebra quasi-isomorphism.  Since $\K$ and $S$ are bounded above and flat over $Z$ in each degree, the reduction $k\ot_Z S\to k\ot_Z \K$ remains a quasi-isomorphism and we have an isomorphism of algebras
\[
\Sym(\Sigma m_Z/m_Z^2)\cong H^\ast(k\ot_ZS)\overset{\cong}\to H^\ast(k\ot_Z \K).
\]
\par

Note that the dg subalgebra $\Sym(\Sigma V_1)\subset k\ot_Z \K$ consists entirely of cocycles, and furthermore $Z^{-1}(k\ot_Z \K)=\Sigma V_1$.  We see also that the intersection $V_1\cap m_Z^2$ consists entirely of coboundaries, as such vectors $v$ lift to cocycles in the acyclic complex $\K$ which are of the form $v+m_Z\ot V_1$.  A dimension count now implies that the projection
\[
V_1=Z^{-1}(k\ot_Z \K)\to H^{-1}(k\ot_Z \K)
\]
reduces to an isomorphism $V_1/(m_Z^2\cap V_1)=H^1(k\ot_Z K)$.  Hence, for the degree $-1$ coboundaries in $k\ot_Z\K$, we have $B^{-1}=V_1\cap m_Z^2$.  One now consults the diagram
\[
\xymatrix{
\Sym(\Sigma m_Z/m_Z^2)\ar[d]_\cong\ar[r]^{\rm incl} & \Sym(\Sigma V_1)\ar[d]\ar[r]^(.25){\rm proj} & \Sym(\Sigma V_1)/(B^{-1})\cong \Sym(\Sigma m_Z/m_Z^2)\ar[dl]\\
H^\ast(k\ot_Z S)\ar[r]^\cong  & H^\ast(k\ot_Z \K),
}
\]
to see that the intersection $B^\ast(k\ot_Z\K)\cap \Sym(\Sigma V_1)$ is necessarily the ideal $(B^1)$ generated by the degree $-1$ coboundaries.  So we find that the projection
\[
f:k\ot_Z \K\to \Sym(\Sigma V_1)/(B^1)\cong \Sym(\Sigma m_Z/m_Z^2)
\]
which annihilates (the images of) all cells $\Sigma V_i$ with $i>1$ is an equivariant dg algebra map, and furthermore an equivariant dg algebra quasi-isomorphism.
\end{proof}

In the following $Z$ a commutative $G$-algebra which is isomorphic to a power series in finitely many variables, as above.

\begin{definition}
An \emph{equivariant Koszul resolution} of $k$ over $Z$ is a $G$-equivariant dg $Z$-algebra $\K_Z$ which is finite and flat over $Z$, comes equipped with an equivariant dg algebra quasi-isomorphism $\epsilon:\K_Z\overset{\sim}\to k$, and also comes equipped with an equivariant dg map $\pi:\K_Z\to \Sym(\Sigma m_Z/m_Z^2)$ which reduces to a quasi-isomorphism $k\ot_Z\K_Z\overset{\sim}\to \Sym(\Sigma m_Z/m_Z^2)$ along the augmentation $Z\to k$.
\end{definition}

Lemma~\ref{lem:equiv_kos} says that equivariant Koszul resolutions of $k$, over such $Z$, always exists.

\subsection{The Koszul resolution associated to an equivariant deformation}
\label{sect:KQ}

Consider $Q\to R$ an equivariant deformation, parameterized by a formally smooth space $\Spf(Z)$, as in Section \ref{sect:KZ}.  For any equivariant Koszul resolution $\K_Z\overset{\sim}\to k$ over $Z$, the product
\begin{equation}\label{eq:205}
\K_Q:=Q\ot_Z \K_Z
\end{equation}
is naturally a dg $G$-algebra which is a finite and flat extension of $Q$.  Since finite flat modules over $Z$ are in fact free, $\K_Q$ is more specifically \emph{free} over $Q$ in each degree.  Flatness of $Q$ over $Z$ implies that the projection
\[
id_Q\ot_Z\epsilon:\K_Q\overset{\sim}\to Q\ot_Z k=R
\]
is a quasi-isomorphism of dg $G$-algebras (cf.~\cite[Section 5.2]{arkhipovbezrukavnikovginzburg04},~\cite[Section 3]{bezrukavnikovginzburg07}, \cite[Section 2]{avramovsun98}).  We call the dg algebra~\eqref{eq:205}, deduced from a particular choice of equivariant Koszul resolution for $Z$, the (or \emph{a}) Koszul resolution of $R$ associated to the equivariant deformation $Q\to R$.

\section{Deformations associated to group embeddings}

Consider now $G$ a \emph{finite} group scheme, and a closed embedding of $G$ into a smooth affine algebraic group $\mcl{H}$.  (We mean specifically a map of group schemes $G\to \mcl{H}$ which is, in addition, a closed embedding.)  We explain in this section how such an embedding $G\to \mcl{H}$ determines an equivariant deformation $\O\to \O(G)$ which fits into the general framework of Section \ref{sect:equiv1}.
\par

Note that such closed embeddings $G\to \mcl{H}$ always exists for finite $G$.  For example, if we choose a faithful $G$-representation $V$ then the corresponding action map $G\to \operatorname{GL}(V)$ is a closed embedding of $G$ into the associated general linear group.

\subsection{The quotient space}

For any embedding $G\to \mcl{H}$ of $G$ into smooth $\mcl{H}$ we consider the quotient space $\mcl{H}/G$.  The associated quotient map $\mcl{H}\to\mcl{H}/G$ is $G$-equivariant, where we act on $\mcl{H}$ via the adjoint action and on $\mcl{H}/G$ via translation.
\par

This is all clear geometrically, but let us consider this situation algebraically.  Functions on the quotient $\O(\mcl{H}/G)$ are the right $G$-invariants $\O(\mcl{H})^G$ in $\O(\mcl{H})$, or rather the left $\O(G)$-coinvariants.  Then $\O(\mcl{H}/G)$ is a right $\O(\mcl{H})$-coideal subalgebra in $\O(\mcl{H})$, in the sense that the comultiplication on $\O(\mcl{H})$ restricts to a coaction
\[
\rho:\O(\mcl{H}/G)\to \O(\mcl{H}/G)\ot\O(\mcl{H}).
\]
We project along $\O(\mcl{H})\to \O(G)$ to obtain the translation coaction of $\O(G)$ on $\O(\mcl{H}/G)$.  The left translation coaction of $\O(G)$ on $\O(\mcl{H})$ restricts to a trivial coaction on $\O(\mcl{H}/G)$.  So, $\O(\mcl{H}/G)$ is a sub $\O(G)$-bicomodule in $\O(\mcl{H})$.
\par

We consider the dual \emph{action} of the group ring $kG=\O(G)^\ast$ on $\O(\mcl{H})$, and find that the inclusion $\O(\mcl{H}/G)\to \O(\mcl{H})$ is an inclusion of $G$-algebras, where we act on $\O(\mcl{H})$ via the adjoint action and on $\O(\mcl{H}/G)$ by translation.  We have the following classical result, which can be found in~\cite[Proposition 5.25 and Corollary 5.26]{milne17}.

\begin{theorem}
Consider a closed embedding $G\to \mcl{H}$ of a finite group scheme into a smooth algebraic group $\mcl{H}$.  The algebra of functions $\O(\mcl{H})$ is finite and flat over $\O(\mcl{H}/G)$, and $\O(\mcl{H}/G)$ is a smooth $k$-algebra.
\end{theorem}

\subsection{The associated equivariant deformation sequence}
\label{sect:embed_def}

Consider $G\to \mcl{H}$ as above and let $1\in \mcl{H}/G$ denote the image of the identity in $\mcl{H}$, by abuse of notation.  We complete the inclusion $\O(\mcl{H}/G)\to \O(\mcl{H})$ at the ideal of definition for $G$ to get a finite flat extension $\hat{\O}_{\mcl{H}/G}\to\hat{\O}_\mcl{H}$.  Take
\[
Z=\hat{\O}_{\mcl{H}/G}\ \ \text{and}\ \ \O=\hat{\O}_\mcl{H}.
\]
So we have the deformation $\O\to \O(G)$, with formally smooth parametrizing algebra $Z$.  A proof of the following Lemma can be found at \cite[Lemma 2.10]{negronpevtsova}.

\begin{lemma}
The completion $\O=\hat{\O}_\mcl{H}$ is Noetherian and of finite global dimension.
\end{lemma}

Note that the ideal of definition for $G$ is the ideal $\mfk{m}\O(G)$, where $\mfk{m}\subset \O(\mcl{H}/G)$ is associated to the closed point $1\in \mcl{H}/G$.

\begin{proposition}\label{prop:eq_def_seq}
Consider a closed embedding $G\to \mcl{H}$ of a finite group scheme into a smooth algebraic group $\mcl{H}$.  Take $\O=\hat{\O}_{\mcl{H}}$ and $Z=\hat{\O}_{\mcl{H}/G}$, where we complete at the augmentation ideal $\mfk{m}$ in $\O(\mcl{H}/G)$.  Then
\begin{enumerate}
\item[(a)] the quotients $\O(\mcl{H}/G)/\mfk{m}^n$ and $\O(\mcl{H})/\mfk{m}^n\O(\mcl{H})$ inherit $G$-algebra structures from $\O(\mcl{H}/G)$ and $\O(\mcl{H})$ respectively.
\item[(b)] The completions $Z$ and $\O$ inherit unique continuous $G$-actions so that the inclusions $\O(\mcl{H}/G)\to Z$ and $\O(\mcl{H})\to \O$ are $G$-linear.
\item[(c)] Under the actions of {\rm (b)}, the projection $\O\to \O(G)$ is an equivariant deformation of $\O(G)$ parametrized by $\Spf(Z)=(\mcl{H}/G)^{\wedge}_1$.
\end{enumerate}
\end{proposition}

\begin{proof}
All of (a)--(c) will follow if we can simply show that $\mfk{m}\subset \O(\mcl{H}/G)$ is stable under the translation action of $kG$.  This is clear geometrically, and certainly well-known, but let us provide an argument for completeness.  If we let $\ker(\epsilon)\subset \O(\mcl{H})$ denote the augmentation ideal, we have $\mfk{m}=\ker(\epsilon)\cap\O(\mcl{H}/G)$.
\par

For the adjoint coaction $\rho_{\rm ad}:f\mapsto f_2\ot S(f_1)f_3$ of $\O(\mcl{H})$ on itself, and $f\in \ker(\epsilon)$, we have
\[
\begin{array}{l}
(\epsilon\ot 1)\circ\rho_{\rm ad}(f)=\epsilon(f_2)S(f_1)f_3\\
\hspace{2cm}=S(f_1)(\epsilon(f_2)f_3)=S(f_1)f_2=\epsilon(f)=0.
\end{array}
\]
So we see that under the adjoint coaction $\rho_{\rm ad}(\ker(\epsilon))\subset \ker(\epsilon)\ot \O(\mcl{H})$.  It follows that $\ker(\epsilon)$ is preserved under the adjoint coaction of $\O(G)$, and hence the adjoint action of $kG$, as well.  So, the intersection $\mfk{m}=\O(\mcl{H}/G)\cap \ker(\epsilon)$ is an intersection of $G$-subrepresentations in $\O(\mcl{H})$, and hence $\mfk{m}$ is stable under the action of $kG$.
\end{proof}

\section{Equivariant formality results and deformation classes}
\label{sect:equiv2}

We observe cohomological implications of the existence of a (smooth) equivariant deformation, for a given finite-dimensional $G$-algebra $R$.  The main results of this section can been seen as particular equivariantizations of~\cite[Theorem 1.2.3]{bezrukavnikovginzburg07} and~\cite[Corollary 4.7]{negronpevtsova}, as well as of classical results of Gulliksen~\cite[Theorem 3.1]{gulliksen74}.

\subsection{We fix an equivariant deformation}
\label{sect:fix}

We fix a $G$-equivariant deformation $Z\to Q\to R$, with $Z$ isomorphic to a power series in finitely many variables.  Fix also a choice of equivariant Koszul resolution
\[
\K:=\K_Z,\ \ \text{with}\ \epsilon:\K\overset{\sim}\to k\ \text{and}\ \pi:\K\to \Sym(\Sigma m_Z/m_Z^2).
\]
Recall the associated dg resolution $K_Q\overset{\sim}\to R$, with $\K_Q=Q\ot_Z\K$.  Via general phenomena (Section \ref{sect:dg_htop}) we observe

\begin{lemma}\label{lem:433}
Restriction provides a derived equivalence $D_{fin}(R)^G\overset{\sim}\to D_{coh}(\K_Q)^G$.
\end{lemma}

Following the notation of \cite{negronpevtsova}, we fix
\begin{equation}\label{eq:a_z}
A_Z:=\Sym(\Sigma^{-2}T_1\Spf(Z))=\Sym(\Sigma^{-2}(m_Z/m_Z^2)^\ast).
\end{equation}

\subsection{Equivariant formality and deformation classes}

\begin{lemma}\label{lem:formal}
Consider $\K$ the regular dg $\K$-bimodule.  There is a ($G$-)equivariant homotopy isomorphism
\[
\REnd_{\K\ot_Z \K}(\K)\overset{\sim}\to A_Z.
\]
In particular, $\REnd_{\K\ot_Z\K}(\K)$ is equivariantly formal.
\end{lemma}

\begin{proof}
Consider our algebra $A=A_Z$ from \eqref{eq:a_z} and take $B=\Sym(\Sigma m_Z/m_Z^2)$.  Let $F\to k$ be the standard resolution of the trivial module over $B$.  The resolution $F$ is of the form $B\ot A^\ast$, as a graded space, with differential given by right multiplication by the identity element $\sum_i x_i\ot x^i$ in $B^{-1}\ot A^2$, and so $F$ admits a natural dg $(B,A)$-bimodule structure.  The action map for $A$ now provides an equivariant quasi-isomorphism $A\overset{\sim}\to \End_{B}(F)=\REnd_{B}(k)$.
\par

For the Koszul resolution $\K$ over $Z$, we have the equivariant quasi-isomorphism $\pi\ot_Z\epsilon:\K\ot_Z \K\overset{\sim}\to  B$ and corresponding restriction and base change equivalences $D(\K\ot_Z\K)^G\leftrightarrows D(B)^G$, which are mutually inverse.  Restriction sends the trivial representation $k$ over $B$ to the regular $\K$-bimodule $k\cong \mcl{K}$.  Hence the base change $B\ot^{\rm L}_{\K\ot_Z\K}\K$ is isomorphic to $k$.  We then get then an equivariant quasi-isomorphism
\[
B\ot^{\rm L}_{\K\ot_Z\K}-:\REnd_{\K\ot_Z\K}(\K)\overset{\sim}\to \REnd_{B}(B\ot^{\rm L}_{\K\ot_Z\K}\K),
\]
with the latter algebra homotopy isomorphic to $\REnd_{B}(k)\cong A$ by Lemma~\ref{lem:well_def}.
\end{proof}

\begin{remark}
In odd characteristic, one can replace the quasi-isomorphism $\pi\ot_Z\epsilon:\K\ot_Z\K\to B$ with the more symmetric map
\[
mult(\frac{1}{2}\pi\ot_Z\frac{-1}{2}\pi):\K\ot_Z\K\to B.
\]
The point is to provide an equivariant quasi-isomorphism which is a retract of the non-equivariant quasi-isomorphism $B\to \K\ot_Z\K$ implicit in~\cite[Lemma 2.4.2]{bezrukavnikovginzburg07}.
\end{remark}

Recall that we are considering an equivariant deformation $Q\to R$, with associated dg resolution $K_Q\overset{\sim}\to R$, as in Section \ref{sect:KQ}.  We have the natural action of $A_Z$ on $D_{coh}(K_Q)$ \cite[\S 3.4]{negronpevtsova}, which is expressed via the algebra map
\begin{equation}\label{eq:500}
A_Z=\End^\ast_{D(\K\ot_Z\K)}(\K)\to Z(D_{coh}(\K_Q))
\end{equation}
to the center of the derived category $Z(D_{coh}(\K_Q))=\oplus_i\Hom_{\text{Fun}}(id,\Sigma^i)$.  Specifically, for any endomorphism $f:\K\to \Sigma^n\K$ in the derived category of $Z$-central bimodules, and $M$ in $D_{coh}(K_Q)$, we have the induced endomorphism
\[
f\ot_{\K}^{\rm L}M:M\to \Sigma^n M.
\]
\par

Suppose, for convenience, that $Q$ is of finite global dimension.  We lift the maps
\begin{equation}\label{eq:ast}
-\ot^{\rm L}_\K M:\End^\ast_{D(\K\ot_Z\K)}(\K)\to \End^\ast_{D(\K_Q)}(M)
\end{equation}
to a dg level, for \emph{equivariant} $M$, as follows \cite{bezrukavnikovginzburg07}: Fix an equivariant semi-projective resolution $F\to \K$ over $\K\ot_Z\K$ and, at each $M$, chose an equivariant quasi-isomorphism $M'\to M$ from a dg $\K_Q$-module which is bounded and projective over $Q$ in each degree.  (Such a resolution exists since $Q$ is of finite global dimension.)  Then $F\ot_{\K}M'\to M$ is an equivariant semi-projective resolution of $M$ over $\K_Q$ \cite[Lemma 4.4]{negronpevtsova}.  We now have the lift
\[
-\ot_\K M':\End_{\K\ot_Z\K}(F)\to \End_{\K_Q}(F\ot_\K M')
\]
of \eqref{eq:ast}, and we write this lift simply as
\[
\mfk{def}^G_M:\REnd_{\K\ot_Z\K}(\K)\to \REnd_{\K_Q}(M).
\]
Direct calculation verifies that $\mfk{def}^G_M$, constructed in this manner, is in fact $G$-linear.  The following result is an equivariantization of~\cite[Corollary 4.7]{negronpevtsova}.

\begin{theorem}\label{thm:equiv_fg}
Consider a $G$-equivariant deformation $Q\to R$, with $R$ finite-dimensional, $Q$ of finite global dimension, and parametrization algebra $Z$ isomorphic to a power series in finitely many variables.  Let $\msc{R}$ denote the formal dg algebra $\REnd_{\K\ot_Z\K}(\K)$ (Lemma~\ref{lem:formal}).
\par

For any $M$ in $D_{coh}(\K_Q)^G$, the equivariant dg algebra map $\mfk{def}^G_M:\msc{R}\to \REnd_{\K_Q}(M)$ defined above has the following properties:
\begin{enumerate}
\item The induced map on cohomology $H^\ast(\mfk{def}^G_M):A_Z\to \End^\ast_{D(\K_Q)}(M)$ is a finite morphism of graded $G$-algebras.
\item For any $N$ in $D_{coh}(\K_Q)^G$, the induced action of $\msc{R}$ on $\RHom_{\K_Q}(M,N)$ is such that
\[
\RHom_{\K_Q}(M,N)\in D_{coh}(\msc{R})^G.
\]
\end{enumerate}
\end{theorem}

By $D_{coh}(\msc{R})^G$ we mean the category of $G$-equivariant dg modules over $\msc{R}$ with finitely generated cohomology over $A_Z=H^\ast(\msc{R})$.  

\begin{proof}
The map $\mfk{def}^G_M$ was already constructed above.  We just need to verify the implications for cohomology, which actually have nothing to do with the $G$-action.  We note that the cohomology $H^\ast(\mfk{def}^G_M)$ is, by construction, obtained by evaluating the functor
\[
-\ot^{\rm L}_{\K}M:D(\K\ot_Z\K)\to D(\K_Q)
\]
at the object $\K$.  (Again, we forget about equivariance here.)  We can factor this functor through the category of $\K_Q$-bimodules
\[
D(\K\ot_Z\K)\overset{-\ot_Z^{\rm L}Q}\longrightarrow D(K_Q\ot K_Q)\overset{-\ot^{\rm{L}}_{\K_Q}M}\longrightarrow D(\K_Q)
\]
to see that the corresponding map to the center~\eqref{eq:500} agrees with that of\ \cite[(3.1.5)]{bezrukavnikovginzburg07} \cite[Section 3.4]{negronpevtsova}.  So the finiteness claims of (1) and (2) follow from~\cite[Corollary 4.7]{negronpevtsova}.
\end{proof}

Via Lemma~\ref{lem:formal} we may replace $D(\msc{R})^G$ with $D(A_Z)^G$, and view $\RHom_{K_Q}$, or equivalently $\RHom_R$, as a functor to $D(A_Z)^G$.  Alternatively, we could work with the dg scheme (shifted affine space) $\mcl{T}^\ast=T_1^\ast\Spf(Z)=\Spec(A_Z)$, and view $\RHom_R$ as a functor taking values in the derived category of equivariant dg sheaves on $\mcl{T}^\ast$.
\par

From this perspective, Theorem \ref{thm:equiv_fg} tells us that $\RHom_R$ has image in the subcategory of dg sheaves on $\mcl{T}^\ast$ with coherent cohomology,
\[
\RHom_R:(D_{fin}(R)^G)^{op}\times D_{fin}(R)^G\to D_{coh}(A_Z)^G\cong D_{coh}(\mcl{T}^\ast)^G.
\]

\begin{remark}
We only use the finiteness claims of Theorem~\ref{thm:equiv_fg} in the case in which all of $Z$, $Q$, and $R$ are commutative.  In this case in particular, claims (1) and (2) of Theorem~\ref{thm:equiv_fg} should be obtainable directly from Gulliksen~\cite[Theorem 3.1]{gulliksen74}.
\end{remark}

\begin{remark}
One may compare the above analyses with the formality arguments of~\cite[Sections 5.4--5.8]{arkhipovbezrukavnikovginzburg04}.
\end{remark}

\section{Touz\'{e}-Van der Kallen and derived invariants}

We recall some results of Touz\'{e} and Van der Kallen~\cite{touzevanderkallen10}.  Our aim is to take derived invariants of Theorem~\ref{thm:equiv_fg} to obtain a finite generation result for equivariant extensions $\Hom_{D(R)^G}^\ast$.  We successfully realize this aim via an invocation of \cite{touzevanderkallen10}.  Throughout this section $G$ is a \emph{finite} group scheme.

\subsection{Basics and notations}

For $V$ any $G$-representation we have the standard group cohomology $H^\ast(G,V)=\Ext^\ast_G(\1,V)$.  For more general objects in $D(\Rep(G))$ we adopt a hypercohomological notation.

\begin{notation}
We let $(-)^{\RG}:D(\Rep(G))\to D(Vect)$ denote the derived invariants functor, $(-)^{\RG}=\RHom_G(\1,-)$.  For $M$ in $D(\Rep(G))$ we take
\[
\mbb{H}^\ast(G,M):=H^\ast(M^{\RG}).
\]
\end{notation}

We note that the hypercohomology $\mbb{H}^\ast(G,M)$ is still identified with morphisms $\Hom^\ast_{D(\Rep(G))}(\1,M)$ in the derived category.  Since $G$ is assumed to be finite, we are free to employ an explicit identification
\[
(-)^{\RG}=\Hom_G(Bar_G,-),
\]
where $Bar_G$ is the standard Bar resolution.  For any dg $G$-algebra $S$ the derived invariants $S^{\RG}$ are naturally a dg algebra in $Vect$, and for any equivariant dg $S$-module $M$, $M^{\RG}$ is a dg module over $S^{\RG}$.  (Under our explicit expression of derived invariants in terms of the bar resolution, these multiplicative structures are induced by a dg coalgebra structure on $Bar_G$, see e.g.\ \cite[\S 2.2]{sanada93}.)  We therefore obtain at any dg $G$-algebra a functor
\begin{equation}\label{eq:RG}
(-)^{\RG}:D(S)^G\to D(S^{\RG}).
\end{equation}
\par

The following well-known fact can be proved by considering the hypercohomology $\mbb{H}^\ast(G,S)$ as maps $\1\to \Sigma^nS$ in the derived category. 

\begin{lemma}
If $A$ is a commutative dg $G$-algebra, then the hypercohomology $\mathbb{H}^\ast(G,A)$ is a also commutative.
\end{lemma}

\subsection{Derived invariants and coherence of dg modules}

We have the following result of Touz\'{e} and Van der Kalen.

\begin{theorem}[{\cite[Theorems 1.4 \& 1.5]{touzevanderkallen10}}]\label{thm:tv1}
Consider $G$ a finite group scheme, and $A$ a commutative $G$-algebra which is of finite type over $k$.  Then the cohomology $H^\ast(G,A)$ is also of finite type, and for any finitely generated equivariant $A$-module $M$, the cohomology $H^\ast(G,M)$ is a finite module over $H^\ast(G,A)$.
\end{theorem}

One can derive this results to obtain

\begin{theorem}\label{thm:tv2}
Consider $G$ a finite group scheme, and $S$ a dg $G$-algebra which is equivariantly formal and has commutative, finite type, cohomology.  Suppose additionally that the cohomology of $S$ is bounded below.  Then the derived invariants functor \eqref{eq:RG} restricts to a functor
\[
(-)^{\RG}:D_{coh}(S)^G\to D_{coh}(S^{\RG}).
\]
Equivalently, for any equivariant dg $S$-module $M$ with finitely generated cohomology over $H^\ast(S)$, the hypercohomology $\mbb{H}^\ast(G,M)$ is finite over $\mbb{H}^\ast(G,S)$.
\end{theorem}

\begin{proof}
Take $A=H^\ast(S)$.  We are free to view, momentarily, $A$ as a non-dg object.  We have that $A$ is finite over its even subalgebra $A^{ev}$, which is a commutative algebra in the classical sense, so that Theorem \ref{thm:tv1} implies that cohomology $H^\ast(G,-)$ sends $A$ to a finite extension of $H^\ast(G,A^{ev})$, and any finitely generated $A$-module to a finitely generated $H^\ast(G,A^{ev})$-module.  Hence $H^\ast(G,A)$ is of finite type over $k$, and $H^\ast(G,N)$ is finite over $H^\ast(G,A)$ for any finitely generated, equivariant, non-dg, $A$-module $N$.
\par

Since $G$ is a finite group scheme, $A$ is also a finite module over its (usual) invariant subalgebra $A^G$, and any $A$-module is finitely generated over $A$ if and only if it is finitely generated over $A^G$.  Theorem \ref{thm:tv1} then tells us that, for any finitely generated $A$-module $N$, the cohomology $H^\ast(G,N)$ is finitely generated over $H^\ast(G,A^G)=H^\ast(G,\1)\ot A^G$, where $H^\ast(G,A^G)$ acts through the algebra map
\[
H^\ast(G,{\rm incl}):H^\ast(G,A^G)\to H^\ast(G,A).
\]
\par

Consider now any dg module $M$ in $D_{coh}(S)^G$.  Formality implies an algebra isomorphism $S\cong A$ in $D(\Rep(G))$ and so identifies $\mbb{H}^\ast(G,S)$ with $\mbb{H}^\ast(G,A)=H^\ast(G,A)$.  We want to show that, for such a dg module $M$, the hypercohomology $\mbb{H}^\ast(G,M)$ is a finitely generated module over $\mbb{H}^\ast(G,S)\cong H^\ast(G,A)$.  It suffices to show that $\mbb{H}^\ast(G,M)$ is finite over $H^\ast(G,A^G)=H^\ast(G,\1)\ot A^G$.  We have the first quadrant spectral sequence (via our bounded below assumption)
\[
E_2^{\ast,\ast}=H^\ast(G,H^\ast(M))\ \Rightarrow\ \mbb{H}^\ast(G,M),
\]
and the $E_2$-page is finite over $H^\ast(G,A^G)$ by the arguments given above.  Since $H^\ast(G,A^G)$ is Noetherian, it follows that the associated graded module $E_\infty^{\ast,\ast}=\operatorname{gr}\mbb{H}^\ast(G,M)$ is finite over $H^\ast(G,A^G)$, and since the filtration on $\mbb{H}^\ast(G,M)$ is bounded in each cohomological degree it follows that the hypercohomology $\mbb{H}^\ast(G,M)$ is indeed finite over $H^\ast(G,A^G)\subset \mbb{H}^\ast(G,S)$~\cite[Lemma 1.6]{friedlandersuslin97}.
\end{proof}

\section{Finite generation of cohomology for Drinfeld doubles}

Consider $G$ a finite group scheme.  Fix a closed embedding $G\to \mcl{H}$ into a smooth algebraic group $\mcl{H}$, and fix also the associated $G$-equivariant deformation
\[
Z\to \O\to \O(G),\ \ Z=\hat{\O}_{\mcl{H}/G},\ \O=\hat{\O}_{\mcl{H}},
\]
as in Section~\ref{sect:embed_def}.  Here $kG$ acts on $\O(G)$ and $\O$ via the adjoint action, and this adjoint action restricts to a translation action on $Z$.  We recall that the embedding dimension of $G$ is the minimal dimension of such smooth $\mcl{H}$ admitting a closed embedding $G\to \mcl{H}$.
\par

We consider the tensor category
\[
Z(\rep(G))\cong \rep(D(G))\cong \Coh(G)^G
\]
of representations over the Drinfeld double of $G$, aka the Drinfeld center of $\rep(G)$.  We prove the following below.

\begin{theorem}\label{thm:ZG}
For any finite group scheme $G$, the Drinfeld center $Z(\rep(G))$ is of finite type and of bounded Krull dimension
\[
\Kdim Z(\rep(G))\leq \Kdim\rep(G)+\operatorname{embed.dim}(G).
\]
\end{theorem}

One can recall our definition of a finite type tensor category, and of the Krull dimension of such a category, from the introduction.  For $\mcl{T}^\ast$ the cotangent space $T^\ast_1\Spf(Z)$, considered as a variety with a linear $G$-action, we show in particular that there is a finite map of schemes $\Spec \Ext^\ast_{Z(\rep(G))}(\1,\1)\to (G\setminus \mcl{T}^\ast)\times \Spec H^\ast(G,\1)$.

\subsection{Preliminaries for Theorem~\ref{thm:ZG}: Derived maps in $Z(\rep(G))$}

We let $G$ act on itself via the adjoint action, and have $\Coh(G)^G=\rep(\O(G))^G$.  The unit object $\1\in \Coh(G)^G$ is the residue field of the fixed point $1:\Spec(k)\to G$.  We have
\[
\REnd_{\Coh(G)^G}(\1)=\REnd_{\Coh(G)}(\1)^{\RG},
\]
as an algebra, and for any $V$ in $\Coh(G)^G$ we have
\[
\RHom_{\Coh(G)^G}(\1,V)=\RHom_{\Coh(G)}(\1,V)^{\RG},
\]
as a dg $\REnd_{\Coh(G)^G}(\1)$-module.
\par

One can observe these identifications essentially directly, by noting that for the projective generator $\O(G)\rtimes kG$ we have an identification of $G$-representations
\[
\Hom_{\Coh(G)}(\O(G)\rtimes kG,V)=\Hom_k(kG,V)=\O(G)\ot V,
\]
and $\O(G)\ot V$ is an injective over $kG$ for any $V$.  Hence the functor $\Hom_{\Coh(G)}(-,V)$ sends projectives objects in $\Coh(G)^G$ to injectives in $\Rep(G)$, and for a projective resolution $F\to \1$ we have identifications in the derived category of vector spaces
\[
\begin{array}{rl}
\RHom_{\Coh(G)^G}(\1,V)&=\Hom_{\Coh(G)^G}(F,V)\\
&=\Hom_{\Coh(G)}(F,V)^G\\
&\cong\Hom_{\Coh(G)}(F,V)^{\RG}=\RHom_{\Coh(G)}(\1,V)^{\RG}
\end{array}
\]
and
\[
\REnd_{\Coh(G)^G}(\1,\1)=\End_{\Coh(G)}(F)^G\cong\End_{\Coh(G)}(F)^{\RG}=\REnd_{\Coh(G)}(\1)^{\RG}.
\]
The middle identification for derived endomorphisms comes from the diagram
\[
\xymatrix{
\End(F)^G\ar[r]\ar[d]_\sim & \End(F)^{\RG}\ar[d]^\sim\\
\Hom(F,\1)^G\ar[r]_\sim & \Hom(F,\1)^{\RG}.
}
\]

\subsection{Proof of Theorem~\ref{thm:ZG}}

\begin{proof}
Fix an embedding $G\to \mcl{H}$ and associated equivariant deformation $\O\to \O(G)$ as above, and take $A=A_Z=\Sym(\Sigma^{-2}(m_Z/m_Z^2)^\ast)$, as in \eqref{eq:a_z}.  Take also $\msc{R}$ the dg $G$-algebra $\REnd_{\mcl{K}_Z\ot_Z\mcl{K}_Z}(\mcl{K}_Z)$.  We recall from Lemma~\ref{lem:formal} that $\msc{R}$ is equivariantly formal, and so homotopy isomorphic to $A$.  We adopt the abbreviated notations $\RHom=\RHom_{\Coh(G)}$ and $\REnd=\REnd_{\Coh(G)}$ when convenient.
\par

We consider the equivariant dg algebra map
\[
\mfk{def}^G_\1:\msc{R}\to \REnd_{\Coh(G)}(\1)
\]
of Theorem~\ref{thm:equiv_fg}, and the action of $\msc{R}$ on each $\REnd_{\Coh(G)}(\1,V)$ through $\mfk{def}^G_\1$.  By Theorems \ref{thm:equiv_fg} and \ref{thm:tv2}, the hypercohomology $\mbb{H}^\ast(G,\REnd(\1))$ is a finite algebra extension of $\mbb{H}^\ast(G,\msc{R})$, and $\mbb{H}^\ast(G,\RHom(\1,V))$ is a finitely generated module over $\mbb{H}^\ast(G,\msc{R})$ for any $V$ in $\Coh(G)^G$.  In particular, $\mbb{H}^\ast(G,\RHom(\1,V))$ is finite over $\mbb{H}^\ast(G,\REnd(\1))$.
\par

Since $\mbb{H}^\ast(G,\msc{R})\cong \mbb{H}^\ast(G,A)$ is of finite type over $k$, by Touz\'{e}-Van der Kallen (Theorem~\ref{thm:tv2}), the above arguments imply that
\[
\mbb{H}^\ast(G,\REnd_{\Coh(G)}(\1))=\Ext^\ast_{\Coh(G)^G}(\1,\1)
\]
is a finite type $k$-algebra, and that each
\[
\mbb{H}^\ast(G,\RHom_{\Coh(G)}(\1,V))=\Ext^\ast_{\Coh(G)^G}(\1,V)
\]
is a finitely generated module over this algebra, for $V$ in $\Coh(G)^G$.  That is to say, the tensor category $Z(\rep(G))\cong\Coh(G)^G$ is of finite type over $k$.
\par

As for the Krull dimension, $\mbb{H}^\ast(G,A)$ is finite over $H^\ast(G,A^G)=H^\ast(G,\1)\ot A^G$, by Touz\'{e}-Van der Kallen, so that
\[
\begin{array}{rl}
\Kdim Z(\rep(G))&=\Kdim \Ext^\ast_{Z(\rep(G))}(\1,\1)\\
&\leq \Kdim H^\ast(G,k)\ot A^G\\
&\hspace{.5cm}=\Kdim H^\ast(G,k)\ot A\\
&\hspace{.5cm}=\Kdim \rep(G)+\dim \mcl{H}/G=\Kdim\rep(G)+\dim\mcl{H}.
\end{array}
\]
When $\mcl{H}$ is taken to be of minimal possible dimension we find the proposed bound,
\[
\Kdim(Z(\rep(G)))\leq \Kdim\rep(G)+\operatorname{embed.dim}(G).
\]
\end{proof}

\bibliographystyle{abbrv}

\end{document}